\documentclass[a4paper]{amsart}

\usepackage[utf8]{inputenc}
\usepackage{amsmath}
\usepackage{amsfonts}
\usepackage{xypic}
\usepackage{amsthm}
\usepackage{mathrsfs}
\usepackage{amssymb}
\usepackage[all]{xy}
\xyoption{rotate}
\usepackage{color}
\usepackage{hyperref}
\usepackage{graphicx}
\usepackage{wrapfig}
\usepackage{float}

\graphicspath{ {dibujos/} }

\numberwithin{equation}{section}

\newtheorem{theorem}{Theorem}[section]
\newtheorem{lemma}[theorem]{Lemma}
\newtheorem{proposition}[theorem]{Proposition}
\newtheorem{corollary}[theorem]{Corollary}

\theoremstyle{remark}

\newtheorem{innercustomgeneric}{\customgenericname}
\providecommand{\customgenericname}{}
\newcommand{\newcustomtheorem}[2]{%
  \newenvironment{#1}[1]
  {%
   \renewcommand\customgenericname{#2}%
   \renewcommand\theinnercustomgeneric{##1}%
   \innercustomgeneric
  }
  {\endinnercustomgeneric}
}

\newcustomtheorem{customthm}{Theorem}
\newcustomtheorem{customprop}{Proposition}
\newcustomtheorem{customlemma}{Lemma}

\newtheoremstyle{TheoremNum}
        {\topsep}{\topsep}              
        {\itshape}                      
        {}                              
        {\bfseries}                     
        {.}                             
        { }                             
        {\thmname{#1}\thmnote{ \bfseries #3}}
    \theoremstyle{TheoremNum}

\newcommand{\Spec}{\mathrm{Spec}}

\newcommand{\End}{\mathrm{End}}

\newcommand{\Ee}{\mathcal{E}}

\renewcommand{\H}{\mathrm{H}}

\newcommand{\M}{\mathrm{M}}

\newcommand{\V}{\mathrm{V}}

\newcommand{\Higgs}{\mathrm{Higgs}}

\newcommand{\ol}[1]{\overline{#1}}

\newcommand{\CC}{\mathbb{C}}

\newcommand{\PP}{\mathbb{P}}

\newcommand\Quotient[2]{
        \mathchoice
            {
                \text{\raise1ex\hbox{\thinspace $#1$}\Big{/} \lower1ex\hbox{$#2$} \thinspace}%
            }
            {
                #1\,/\,#2
            }
            {
                #1\,/\,#2
            }
            {
                #1\,/\,#2
            }
    }

\newcommand\GIT[2]{
        \mathchoice
            {
                \text{\raise1ex\hbox{\thinspace $#1$}\Big{/}\!\!\!\!\Big{/} \lower1ex\hbox{$#2$} \thinspace}%
            }
            {
                #1\,/\,#2
            }
            {
                #1\,/\,#2
            }
            {
                #1\,/\,#2
     a       }
    }

\title[]{\bf Very stable bundles\\ and properness of the Hitchin map}

\author{Christian \textsc{Pauly}}
\address{Christian Pauly \\ Laboratoire de Math\'ematiques J.A. Dieudonn\'e \\ UMR  7351 CNRS \\ Universit\'e de Nice Sophia-Antipolis \\ 06108 Nice Cedex 02, France}
\email{pauly@unice.fr}

\author{Ana Pe\'on-Nieto}
\address{Ana Pe\'on-Nieto \\ Laboratoire de Math\'ematiques J.A. Dieudonn\'e \\ UMR  7351 CNRS \\ Universit\'e de Nice Sophia-Antipolis \\ 06108 Nice Cedex 02, France}
\email{ana.peon-nieto@unice.fr}

\thanks{The second author was supported by a
post-doctoral grant associated to the Marie Curie project GEOMODULI of the programme FP7/PEOPLE/2013/CIG, project number 618471.}

\date{\today}


\begin{document}
 
\maketitle

\begin{abstract}
Let $X$ be a smooth complex projective curve of genus $g\geq 2$ and let $K$ be its canonical bundle. 
In this note we show that a stable vector bundle
$E$ on $X$ is very stable, i.e.  $E$ has no non-zero nilpotent Higgs field, if and only if the restriction of the Hitchin
map to the vector space of Higgs fields $H^0(X, \End(E) \otimes K)$ is a proper map.
\end{abstract}

\section{Introduction}

Let $X$ be a smooth complex projective curve of genus $g\geq 2$ and let $K$ be its canonical bundle.  
We consider the moduli space $\M_X(n,d)$ of semi-stable 
degree-$d$ rank-$n$ vector bundles on $X$ and the moduli space $\Higgs_X(n,d)$ of semi-stable degree-$d$ rank-$n$ Higgs bundles. Thanks to a result of N. Nitsure \cite{Nitin} it is
known that the Hitchin fibration
$$ h : \Higgs_X(n,d) \longrightarrow  \H = \bigoplus_{i=1}^n \H^0(X, K^i), $$
defined by associating to a Higgs bundle $(E, \phi)$ the coefficients of the characteristic polynomial of $\phi$,
is a proper map. If $E$ is a stable degree-$d$ rank-$n$ vector bundle, the vector space $\V= \H^0(X,\End(E)\otimes K)$ embeds naturally in the moduli space 
$\Higgs_X(n,d)$. So we can consider the restriction $h_V : \V \longrightarrow \H$ of the Hitchin map $h$ to the vector space $\V$ and ask
whether $h_V$ is also proper.

In order to state the answer we need to consider very stable vector bundles introduced by Drinfeld. By definition a vector bundle $E$ is very stable 
if it has no non-zero nilpotent Higgs field. Laumon \cite[Proposition 3.5]{L} proved, assuming $g \geq 2$, that a very stable vector bundle
is stable and that the locus of very stable bundles is a non-empty open subset of $\M_X(n,d)$.

With these notation our main result is the following

\begin{theorem}
Let $E$ be a stable degree-$d$ rank-$n$ vector bundle over $X$. Then we have the following equivalences:
$$
\begin{array}{lcl}
E \textnormal{ \it is very stable } &\iff& \V \textnormal{ \it is closed in } \Higgs_X(n,d)\\
&\iff& h_V \textnormal{ \it  is a proper map}\\
&\iff& h_V \textnormal{ \it  is a quasi-finite map.}
\end{array}
$$
\end{theorem}

A few comments on the proof: the core of the result is to show that for a very stable $E$ the vector 
space $V$ is closed in $\Higgs_X(n,d)$, or equivalently, that if there exists a limit point 
$(F, \psi) \in \Higgs_X(n,d) \setminus V$, then $E$ has a non-zero nilpotent Higgs field. In order to do that 
we use the $\CC^*$-action on the one-dimensional family of Higgs bundles converging to $(F,\psi)$ to contruct 
a rational map from a smooth surface to $\Higgs_X(n,d)$ whose indeterminacy locus is one point. Then, Hironaka's theorem
on the resolution of indeterminacies \cite{H} gives a morphism from the exceptional divisor (a chain of projective lines) to
$\Higgs_X(n,d)$ connecting the two Higgs bundles $(E,0)$ and $(F, \psi)$. 

We have the following
\begin{corollary}
If $E$ is very stable, then the restricted Hitchin map $h_V$ is finite and surjective.
\end{corollary}

We would like to thank T. Hausel, J. Heinloth, C. Simpson and H. Zelaci for useful discussions on these questions.

\bigskip

\section{Preliminaries}

In this section we recall basic definitions and prove some preliminary results used in the proof of Theorem 1.1.

By \cite{S} section 6 there is an algebraic action of $\CC^*$ on the coarse moduli space $\Higgs_X(n,d)$ given by
multiplying the Higgs field by scalars
$$ \lambda \cdot (E, \phi) \mapsto (E, \lambda \cdot \phi).$$
Clearly the subset $V \subset \Higgs_X(n,d)$ is invariant for the $\CC^*$-action.

\begin{proposition}\label{prop q.f. then vs}
Let $E$ be a stable bundle.  If $h_V$ is quasi-finite, then $E$ is very stable.
\end{proposition}
\begin{proof}
Suppose that $E$ is stable, but not very stable, and let $\phi \in \V$ be a non-zero nilpotent Higgs field.
Then $h_V^{-1}(0)$ contains
the line $\CC \phi$, a contradiction.
\end{proof}

\begin{proposition}\label{prop V closed iff hv proper}
Let $E$ be a stable bundle. Then
$$ \V \textnormal{ \it is closed in } \Higgs_X(n,d) \  \iff \  h_V \textnormal{ \it  is a proper map.}$$
\end{proposition}
\begin{proof}
This is a consequence of the valuative criterion of properness applied to the morphism 
$i_V : V \rightarrow \Higgs_X(n,d)$ and its composite $h_V = h \circ i_V$ with the proper map $h$, see e.g.
\cite{Ha} Corollary II.4.8 (a),(b) and (e).
\end{proof}

\begin{proposition}\label{prop M separable}
Let $E$ be a stable bundle and let $C$ be a smooth curve with a morphism 
$$\varphi : C \rightarrow \Higgs_X(n,d),$$
such that $\varphi(C \setminus \{ c \}) \subset V$ for some point $c \in C$. Denote $\varphi(c) = (F, \psi)$. 
If $F\neq E$, then $F$ is not semi-stable.
\end{proposition}

\begin{proof}
Suppose on the contrary that $F$ is semi-stable. By passing to an \'etale cover of $C$ we can assume that there is 
a family of vector bundles $\Ee$ over X parameterized by C such that $\Ee_{|X \times \{ p \}} = E$ for $p \not=c$ and
$\Ee_{|X \times \{ c\}} = F$. Now the classifying map $\varphi'$ associated to $\Ee$ maps $C$ 
to the coarse moduli space $\M_X(n,d)$, which
is separated. Hence $\varphi'$ is constant, which contradicts $F \not= E$.
\end{proof}

The next lemma is probably well-known, but as we have not found a reference in the literature we include a 
full proof.

\begin{lemma} \label{rk zariski closure}
Given a morphism $f : T \rightarrow Y$ to a quasi-projective variety $Y$. Let $U = f(T)$ be the 
image of $f$ in $Y$. Suppose that there exists a point $z \in \overline{U} \setminus {U}$, where $\overline{U}$ is the 
Zariski closure of $U$ in $Y$. Then there exists a smooth (not necessarily complete) curve $C$, a point $c \in C$ and a morphism $\varphi: C \rightarrow \overline{U}$ such that
$$
\varphi( C\setminus \{ c \}) \subset U \ \text{and} \  \varphi(c) = z.
$$
\end{lemma}

\begin{proof}
By Chevalley's theorem (see e.g. \cite{Ha} Ex II.3.18 and Ex II.3.19) we know that the image $U = f(T)$
is a constructible set, hence a finite disjoint union $U= \bigcup_{i} U_i$ of locally closed subsets $U_i$. Hence 
$z \in \overline{U}_i \setminus U_i$ for some integer $i$. To simplify notation we will write
$U$ instead of $U_i$. Clearly $\dim U \geq 1$.

We choose an embedding of $U \subset \overline{U} \subset Y \hookrightarrow \PP^N$ in projective space. 
We denote by $Z$ the irreducible component of $\overline{U} \setminus U$ which contains $z$. Then
$\delta = \dim Z < \dim U$. We choose a point $u \in U$ and consider the linear system 
$\Gamma = | \mathcal{J}_{u,z}(m) |$ of hypersurfaces in $\PP^N$ of fixed degree $m \geq 2$ through the
two points $u$ and $z$. Our strategy is to cut out a curve through the limit point $z$ by 
intersecting divisors in the linear system $\Gamma$. For that we observe that the base locus 
of $\Gamma$ is reduced to $\{ u, z \}$. Therefore we can choose $\delta$ divisors 
$D_1, D_2, \ldots, D_\delta$ in $\Gamma$ which cut out on $Z$  a finite set of points containing $z$.
We denote by $W$ the intersection $D_1 \cap D_2 \cap \ldots \cap D_\delta \cap U \subset U$. Then 
$W$ is non-empty, since $u \in W$, and $\dim W \geq 1$. Moreover, since $W$ is closed in $U$, we have the
inclusion $\overline{W} \setminus W \subset
D_1 \cap D_2 \cap \ldots \cap D_\delta \cap Z$, which shows that $\overline{W} \setminus W$ is also
a finite set of points containing $z$. If $\dim W > 1$ we intersect $W$ with divisors in $\Gamma$ till
we obtain a curve $C$ passing through $z$. Since $\overline{W} \setminus W$ is a finite set, there is a
neighboorhood $\Omega$ of $z$ in $C$ which does not intersect this finite set, hence $\Omega \setminus \{ z \} $
is contained in $U$. If the curve $C$ has singular points, we take its normalization. 
\end{proof}

\section{Proof of Theorem 1.1}
 
Because of Propositions \ref{prop q.f. then vs} and \ref{prop V closed iff hv proper} it will be 
enough to show that:

$i)$ If $h_V$ is proper, then it is quasi-finite.

$ii)$ $V$ is closed in $\Higgs_X(n,d)$ if $E$ is very stable, or equivalently, that if the
Zariski closure $\ol{\V}$ of $V$ in $\Higgs_X(n,d)$ properly contains $V$, then $E$ admits a non-zero
nilpotent Higgs field.

To prove $i)$, note that $h_V$ is a proper map between affine spaces of the same dimension, and it is thus quasi-finite. 

For $ii)$, assume that there exists $(F,\psi)\in Z:=\overline{V}\setminus V$, where $\overline{V}$ is the Zariski closure of $V$ in
$\Higgs_X(n,d)$. Since $V$ is invariant for the $\CC^*$-action, its Zariski closure $\overline{V}$ is also invariant for the
$\CC^*$-action. Therefore the limit point $(F_0, \psi_0) := \lim_{\lambda \to 0} (F, \lambda \psi) \in \overline{V}$ and 
satisfies $h(F_0, \psi_0) = 0$. By Lemma \ref{rk zariski closure} and Proposition \ref{prop M separable} we deduce that $(F, \psi)$ 
and therefore $(F_0, \psi_0)$ are not semi-stable, so $(F_0, \psi_0) \in Z$.
Hence, replacing $(F, \psi)$ by its limit $(F_0, \psi_0)$ for the $\CC^*$-action, we can assume that $h(F,\psi) = 0$ and that 
$(F, \psi)$ is a fixed point for the $\CC^*$-action.
\bigskip

By Lemma \ref{rk zariski closure} there exists a curve $C$, a point $c \in C$ and a 
morphism $\varphi: C \rightarrow \overline{V}$ such that 
$\varphi( C^*) \subset V$ and $\varphi(c) = (F, \psi)$. Here $C^*$ denotes the curve $C \setminus \{c \}$.
The main idea of the proof is to consider the $\CC^*$-action on the image of the curve $C^*$ in $\Higgs_X(n,d)$.
So we introduce the morphism
$$ \Psi^*: \CC^* \times C^* \longrightarrow \overline{\V} \ \ \text{defined by}  \ \ \Psi^* (\lambda, p) = \lambda \cdot \varphi(p). $$
\begin{proposition}
We can extend the morphism $\Psi^*$ to a morphism 
$$ \Psi : \CC \times C \setminus \{ (0,c) \} \rightarrow \overline{V}, $$
such that $\Psi(0,p) = (E,0)$ for $p \not= c$ and $\Psi(\lambda,c) = (F, \psi)$ for $\lambda \not=0$.
\end{proposition}

\begin{proof}
It will be enough to extend $\Psi^*$ to the two open subsets $\CC \times C^*$ and $\CC^* \times C$ of $\CC \times C$.
Note that the $\CC^*$-action $\CC^* \times V \rightarrow V$ is the scalar multiplication of the vector space $V$ and
hence naturally extends to an action $\CC \times V \rightarrow V$. Since $\varphi(C^*) \subset V$, we therefore obtain
a morphism $\Psi : \CC \times C^* \rightarrow V$. Clearly, $\Psi(0,p) = (E,0)$ for $p \not= c$. On the open subset
$\CC^* \times C$, we just take the definition of $\Psi^*$ extended to $C$. Clearly, $\Psi(\lambda, c) = \lambda \cdot (F, \psi) 
= (F, \psi)$ for $\lambda \not=0$, as $(F, \psi)$ is by assumption a fixed point.
\end{proof}

So $\Psi$ is a rational map from the surface $S := \CC \times C$ to $\Higgs_X(n,d)$ whose 
indeterminacy locus is the point $(0,c)$. First, we consider
the rational composite map
$$h' =  h \circ \Psi : S \rightarrow \H.$$
Since $\H$ is a vector space, the  morphism $h'$ is given by holomorphic functions on a punctured smooth surface. By 
Hartog's theorem these functions extend to the surface $S$ and by continuity $h'(0, c) = 0$ since
$h'(0, p) = 0$ for $p \not= c$.

In order to prove that the rational map $\Psi$ can be resolved into a morphism, we will apply 
Main Theorem II in \cite{H}. For it to apply, by the discussion following Question E in loc. cit. (page 140), it is 
enough to prove that the morphism $\pi$ defined in the following commutative diagram is proper

$$
\xymatrix{
\Gamma\ar[dr]_\pi\ar@{^(->}[r]&S \times \Higgs_X(n,d)\ar[d]^{\pi_1}\\
&S.}
$$
Here $\Gamma$ is the closure of the graph of $\Psi$.
Let $R$ be a discrete valuation ring with quotient field $L$. By the valuative criterion of properness, we need to prove 
that for any commutative diagram as below, the dashed arrow exists  
$$
\xymatrix{
\Spec(L)\ar[d]\ar[r]&\Gamma\ar[d]_\pi\\
\Spec(R)\ar@{-->}[ur]\ar[r]& S.
}
$$
For each such commutative diagram, consider the extended commutative diagram 
\begin{equation}\label{eq extend Psi}
\xymatrix{
\Spec(L)\ar[d]\ar[r]&\Gamma\ar[d]_\pi\ar@{^(->}[r]&S \times \Higgs_X(n,d)\ar[d]_{Id \times h}\\
\Spec(R)\ar[r]\ar@{-->}[ur]_{e_1}\ar@{-->}[urr]_{e_2}&S \ar[r]_{Id \times h'}& S \times \H,
}
\end{equation}
where the morphism $h' : S \rightarrow \H$ was introduced above. By properness of the Hitchin map $h$, the map $Id \times h$ is
also proper and therefore the dashed arrow $e_2$ in \eqref{eq extend Psi}  exists. Moreover, since $\Gamma$ is closed, its image is contained in $\Gamma$, so the dashed arrow $e_1$ also exists. So, by \cite{H} Main Theorem II, this proves that $\Psi$ 
resolves after a finite sequence of 
blow-ups along points to a morphism $$\hat{\Psi} : \hat{S} \rightarrow \ol{V}.$$

First note that the exceptional divisor $D:=\bigcup_{i=0}^m D_i$ is a connected union of projective lines $D_i$ and that
by restriction of $\hat{\Psi}$ we obtain a morphism
$$
\hat{\Psi}: D:=\bigcup_{i=0}^mD_i \longrightarrow \ol{V}\subset\Higgs_X(n,d)
$$
whose image is a connected curve in $\ol{V}$. Let $p_0 \in D$ and $p_\infty \in D$ be the limit points in $\hat{S}$
$$ p_0 = \lim_{p \to c} (0,p) \ \text{and} \ p_\infty = \lim_{\lambda \to 0} (\lambda, c).$$
Then, by separability of the moduli space $\Higgs_X(n,d)$ \cite[Theorem 5.10]{Nitin} (see also Remark 5.12 in loc.cit.), we clearly have
$$ \hat{\Psi}(p_0)=(E,0) \ \text{and} \   \hat{\Psi}(p_\infty)=(F,\psi).$$
Moreover, since $h'(0,c) = 0$ we have $\hat{\Psi}(D)\subset h^{-1}(0)$. 

We can numerate the projective lines $D_0, D_1, \dots, D_{m'}$ such that $p_0 \in D_0$ and $p_\infty \in D_{m'}$
and $D_i \cap D_{i+1} \not= \emptyset$ for all $i \leq m'-1$. Note that a priori $m'$ can be smaller than $m$.
Consider the smallest integer $i_0 \geq 0$ such that $\hat{\Psi}(D_{i_0})$ is not reduced to the point
$(E, 0)$. Such an integer exists since $\hat{\Psi}(p_\infty) = (F, \psi) \not= (E, 0)$. Then we claim that
$\hat{\Psi}(D_{i_0}) \cap V$ contains a Higgs bundle $(E, \phi_0)$ with non-zero Higgs field $\phi_0$.
Indeed, by definition of $i_0$ there exists a point $p \in D_{i_0}$ such that $\hat{\Psi}(p) = (E,0)$.
Suppose on the contrary that for any point $q \not=p$ we have $\hat{\Psi}(q) \in \overline{V} \setminus V$. Then by 
Lemma \ref{rk zariski closure} and Proposition \ref{prop M separable} the underlying bundle of $\hat{\Psi}(q)$
is not semi-stable for all $q \not=p$. This contradicts the fact that the non-empty locus of stable
bundles of the family parameterized by $D_{i_0}$ is an open subset.

Since $\hat{\Psi}(D_{i_0}) \subset h^{-1}(0)$ the Higgs field $\phi_0$ is nilpotent showing that $E$ is not 
very stable.

\section{Proof of Corollary 1.2}
If $E$ is a very stable bundle, then by Theorem 1.1 the map $h_V$ is a proper map between affine spaces. Hence the fibers of $h_V$ are affine and complete, so $h_V$ is a quasi-finite map.
But $h_V$ proper and quasi-finite implies that $h_V$ is a finite map. Surjectivity of $h_V$ follows again from properness.

\end{document}